\documentclass[12pt]{amsart}

\usepackage{amssymb, amscd, tikz, latexsym, amsmath, amsthm}
\usepackage{verbatim}
\usepackage{upref}
\usepackage[all]{xy}
\usepackage{color}

\newcommand{\field}[1]{\mathbb{#1}}
\newcommand{\CC}{\field{C}}
\newcommand{\FF}{\field{F}}
\newcommand{\NN}{\field{N}}

\newcommand{\QQ}{\field{Q}}

\newcommand{\Aa}{\mathcal A}
\newcommand{\Bb}{\mathcal B}

\newcommand{\Dd}{\mathcal D}

\newcommand{\Ii}{\mathcal I}
\newcommand{\Kk}{\mathcal K}
\newcommand{\Ll}{\mathcal L}

\newcommand{\Oo}{\mathcal O}

\newcommand{\Tt}{\mathcal T}

\newcommand{\NO}[1]{\operatorname{\mathcal{NO}}_{\!#1}}
\newcommand{\NT}{\operatorname{\mathcal{NT}}}

\newcommand{\reduced}{r}
\newcommand{\normal}{n}

\newcommand{\id}{\operatorname{id}}
\newcommand{\Ind}{\operatorname{Ind}}

\newcommand{\clsp}{\operatorname{\overline{span\!}\,\,}}

\newcommand{\fp}{FESSPE}

\newcommand{\Tc}{\Tt_{\rm cov}}

\theoremstyle{plain}
\newtheorem{theorem}{Theorem}[section]
\newtheorem*{theorem*}{Theorem}
\newtheorem*{prop*}{Proposition}
\newtheorem{cor}[theorem]{Corollary}
\newtheorem{lemma}[theorem]{Lemma}
\newtheorem{prop}[theorem]{Proposition}

\theoremstyle{remark}
\newtheorem{rmk}[theorem]{Remark}

\newtheorem{example}[theorem]{Example}

\theoremstyle{definition}

\numberwithin{equation}{section}

\allowdisplaybreaks
\theoremstyle{plain}
\newtheorem{thm}[theorem]{Theorem}
\newtheorem{lem}[theorem]{Lemma}

\theoremstyle{definition}
\newtheorem{defn}[theorem]{Definition}

\newtheorem{ex}[theorem]{Example}

\newtheorem*{notn*}{Notation}

\newtheorem*{hyp*}{Hypothesis}

\theoremstyle{remark}
\newtheorem{rem}[theorem]{Remark}

\newcommand{\thmref}[1]{Theorem~\textup{\ref{#1}}}
\newcommand{\corref}[1]{Corollary~\textup{\ref{#1}}}
\newcommand{\lemref}[1]{Lemma~\textup{\ref{#1}}}

\newcommand{\midtext}[1]{\quad\text{#1}\quad}

\newcommand{\Chi}{\raisebox{2pt}{\ensuremath{\chi}}}

\DeclareMathOperator{\ad}{Ad}

\newcommand{\case}{& \text{if }}
\newcommand{\ifnot}{& \text{otherwise}}

\renewcommand{\iff}{\ensuremath{\Leftrightarrow}}

\newcommand{\minus}{\setminus}

\newcommand{\inv}{^{-1}}

\renewcommand{\bar}{\overline}

\newcommand{\wilde}{\widetilde}

\begin{document}

\title[Inner coactions, Fell bundles, and abstract uniqueness]{Inner coactions, Fell bundles, and abstract uniqueness theorems}
\author[S. Kaliszewski]{S. Kaliszewski}
\address{School of Mathematical and Statistical Sciences, Arizona State University, Tempe, Arizona, 85287}
\email{kaliszewski@asu.edu}
\author[N.S. Larsen]{Nadia S. Larsen}
\address{Department of Mathematics, University of Oslo, PO BOX 1053 Blindern, N-
0316 Oslo, Norway.}
\email{nadiasl@math.uio.no}
\author[J. Quigg]{John Quigg}
\address{School of Mathematical and Statistical Sciences, Arizona State University, Tempe, Arizona, 85287}
\email{quigg@asu.edu}

\thanks{This research was supported by the Research Council of Norway and the
NordForsk Research Network ``Operator Algebra and Dynamics'' (grant \#11580).}

\subjclass[2010]{Primary 46L05, 46L55}

\keywords{Fell bundle, inner coaction, quasi-lattice ordered group, product system}

\begin{abstract}{We prove gauge-invariant uniqueness theorems with respect to maximal and normal coactions
for $C^*$-algebras associated to product systems of $C^*$-correspondences. Our techniques of proof are developed in the abstract context of
Fell bundles. We employ inner coactions to prove an essential-inner uniqueness
theorem for Fell bundles. As application, we characterise injectivity of homomorphisms on Nica's Toeplitz algebra
$\Tt(G, P)$ of a quasi-lattice ordered group $(G, P)$ in the presence of a finite non-trivial
set of lower bounds for all non-trivial elements in $P$.}
\end{abstract}

\date{August 24, 2011. Revised April 5, 2012}
\maketitle

\section{Introduction}\label{intro}

Starting with the early constructions of $C^*$-algebras associated to generating families of operators on Hilbert space
such as isometries or partial isometries, possibly subject to certain relations,
a question of interest arose as to whether the $C^*$-algebra was unique. Coburn's theorem asserts that the $C^*$-algebra
generated by a non-unitary isometry on Hilbert space is unique up to isomorphism, \cite{Co}. In \cite{Cu}, Cuntz constructed
large classes of $C^*$-algebras, both simple and non-simple, generated by families of isometries satisfying certain relations, and proved that
two tuples of isometries on Hilbert space fulfilling the same relation generate isomorphic $C^*$-algebras.

In a remarkable generalization, Nica introduced the notion of a quasi-lattice ordered group $(G, P)$ and constructed a Toeplitz
$C^*$-algebra $\Tt(G, P)$ and a universal $C^*$-algebra $C^*(G, P)$, \cite{N}. He obtained both analogues of Coburn's theorem, and
results relating to Cuntz's uniqueness theorems in the particular case of the quasi-lattice ordered group
$(\FF_n, \FF_n^+)$ consisting of the free group and the free semigroup on $n$ generators. Laca and Raeburn \cite{LacR1} discovered a
semigroup crossed product structure of $C^*(G, P)$, and used it to prove faithfulness results for representations of this
algebra in the presence of an amenability hypothesis.

Our starting point is two-fold. For one thing, we noticed that the analysis of the gauge-invariant uniqueness
property from \cite{CLSV} involved two crucial ingredients of nonabelian duality, namely \emph{maximal} and \emph{normal} coactions. The
second motivating fact was that the quasi-lattice ordered group $(\FF_n, \FF_n^+)$ belongs to the class of
those $(G, P)$ for which, as Nica showed,
$\Tt(G, P)$ contains $\Kk(l^2(P))$. We could see that for such pairs $(G, P)$, the ideal $\Kk(l^2(P))$ of $\Tt(G, P)$
contains a family of projections that determines an \emph{inner} coaction.

Our thrust in this paper is to show how the general theory of coactions gives uniqueness theorems for
$C^*$-algebras of Fell bundles in a systematic manner. We  apply these results in familiar contexts
with sharpened or new characterizations of uniqueness as outcomes.

The first gauge-invariant uniqueness type result was proved by an Huef and Raeburn in \cite{anHR}.
Here we obtain a gauge-invariant uniqueness result in the context of Fell bundles. Since the proofs of the
abstract gauge-invariant uniqueness results for
$C^*$-algebras of Fell bundles are painless, albeit non-trivial, we chose to place this
material in an appendix. The other type of abstract uniqueness results we prove emerges from inner coactions.

The first main application is to establish a gauge-invariant uniqueness property  for
the Cuntz-Nica-Pimsner algebra $\NO{X}$ of Sims and Yeend
from \cite{SY} by highlighting  the feature observed in \cite{CLSV} that it
carries a maximal coaction. If $(G, P)$ is a quasi-lattice ordered
group and $X$ is a compactly aligned
product system over $P$ of $C^*$-correspondences over a $C^*$-algebra $A$, then Sims and Yeend's
$C^*$-algebra $\NO{X}$  is universal for Cuntz-Nica-Pimsner covariant representations
of $X$. When $X$ is $\tilde{\phi}$-injective, $\NO{X}$ has the desired property of admitting an injective universal Cuntz-Nica-Pimsner
covariant representation. For product systems, $\NO{X}$ is the appropriate candidate for
the Cuntz-Pimsner algebra $\Oo_Y$ associated in \cite{Ka} to
 a single $C^*$-correspondence $Y$, in a generalization of Pimsner's work from \cite{P}.

The gauge-invariant
uniqueness property for $\NO{X}$ proved in \cite{CLSV} (see Corollaries 4.11 and 4.12) is
equivalent to asking for the canonical maximal coaction on $\NO{X}$ to be normal. In our
treatment here we  look at the gauge-invariant uniqueness property in two separate
classes, that of $C^*$-algebras with maximal coactions, and of $C^*$-algebras
with normal coactions. Thereby we are in the context of coactions and
can streamline the proofs by using specific techniques.  We obtain gauge-invariant uniqueness
theorems for $\NO{X}$, seen in the category of maximal coactions, and
for the co-universal algebra $\NO{X}^\reduced$ identified in \cite{CLSV} and viewed in the
category of normal coactions.

As a bonus for sorting out abstract gauge-invariant uniqueness results for Fell bundles,
we also obtain  a
gauge-invariant uniqueness theorem for the Toeplitz-like extension of $\NO{X}$.
This is the universal $C^*$-algebra for Nica covariant Toeplitz representations of
the compactly aligned product system $X$; this algebra was denoted $\Tc(X)$ in \cite{F}, but
we shall follow \cite{BanHLR}, see their Remark 5.3, and use the notation $\NT(X)$.

Faithfulness of representations of $\Tt(G, P)$ was characterized by Laca and Raeburn for all
amenable quasi-lattice ordered groups $(G, P)$, \cite{LacR1}. In coaction terminology, $(G, P)$ amenable means
that the canonical maximal coaction on $C^*(G, P)$ is also normal.  Here we exploit the fact that $\Tt(G, P)$ has a natural normal coaction. For a
quasi-lattice ordered group $(G, P)$ with the property that there is a finite set $F$ of elements in
$P\setminus\{e\}$ such that every non-trivial element in $P$ has a  lower bound in $F$, we characterize
directly injectivity of homomorphisms from $\Tt(G, P)$ to a $C^*$-algebra $B$. The crucial observation is that existence
of $F$ not only characterizes the fact that $\Kk(l^2(P))$ is included in $\Tt(G, P)$, as proved by Nica
in \cite[Proposition 6.3]{N}, but that it also characterizes existence of an inner coaction
on the ideal $\Kk(l^2(P))$. With this card at hand, we can apply our abstract essential-inner uniqueness result, i.e.
Corollary~\ref{cor:eiut}. For the pair $(\FF_n, \FF_n^+)$, which clearly admits a finite set of lower bounds for elements
in $\FF_n^+$, our Theorem~\ref{thm:describe-inner-Nica-alg} thus provides a characterization of faithful representations
of $\Tt(\FF_n, \FF_n^+)$ without reference to the amenability of the pair, a property that is by no means trivial to verify.

The organization of the paper is as follows: after a preliminary section in which we recall terminology and facts about
coactions, quasi-lattice ordered groups and $C^*$-algebras of product systems, in section~\ref{NO} we present gauge-invariant
uniqueness theorems for the Nica-Toeplitz algebra, the Cuntz-Nica-Pimsner algebra and the co-universal $C^*$-algebra
of a class of compactly aligned product systems $X$. In section~\ref{section:inner} we prove the abstract inner-uniqueness and
essential-inner uniqueness results. In section~\ref{subsect:Nica-alg-semigp-cp} we place the representation of
$C^*(G, P)$ arising from the Toeplitz representation of $P$ in the framework of nonabelian duality.
Section~\ref{section:fessub} contains the essential-inner uniqueness theorem for
$\Tt(G, P)$, namely  Theorem~\ref{thm:describe-inner-Nica-alg}, and a converse to it, Theorem~\ref{thm:inner-to-FES}. The appendix
collects the promised gauge-invariant uniqueness results for Fell bundles.

\section{Preliminaries}
\label{prelim}

Throughout, $G$ will be a discrete group.
If $A$ is a $C^*$-algebra and $\delta:A\to A\otimes C^*(G)$ is a coaction, we will just say ``$(A,\delta)$ is a coaction''.
For the theory of coactions we refer to \cite[Appendix~A]{BE}, and for discrete coactions in particular we refer to \cite{EKQ, QuiggDiscrete}.
For maximalizations and normalizations of coactions we refer to \cite{clda, nordfjordeid}.

If $(A,\delta)$ is a (full\footnote{and all our coactions will be full}) coaction of $G$, we will let $\Aa$ denote the associated Fell bundle, and similarly for other capital letters.
If $\pi:(A,\delta)\to (B,\varepsilon)$ is a morphism of coactions, we write $\tilde{\pi}:\Aa \to \Bb$ for the corresponding homomorphism of
Fell bundles. Note that $\pi$ is surjective
if and only if $\{\pi(A_s):s\in G\}$ generates $B$.
Also note that if $(A,\delta)$ and $(B,\varepsilon)$ are coactions, then  a homomorphism $\pi:A\to B$ is $\delta-\varepsilon$ equivariant if and only if $\pi(A_s)\subset B_s$ for all $s\in G$ (because equivariance can be checked on the generators $a_s\in A_s$ for $s\in G$).

A morphism $\pi:(B,\varepsilon)\to (A,\delta)$ of coactions is a \emph{maximalization} of $(A,\delta)$ if $(B,\varepsilon)$ is maximal and $\pi\times G:B\times_\varepsilon G\to A\times_\delta G$ is an isomorphism.
Sometimes we call $(B,\varepsilon)$ itself a maximalization of $(A,\delta)$.
Maximalizations of $(A,\delta)$ always exist, and all are uniquely isomorphic.
Choosing one for every coaction, we get a \emph{maximalization functor} that sends $(A,\delta)$ to the maximalization
\[
q_A^m:(A^m,\delta^m)\to (A,\delta),
\]
and sends a morphism $\pi:(A,\delta)\to (B,\epsilon)$ to the unique morphism $\pi^m$, called the \emph{maximalization} of $\pi$, making the diagram
\[
\xymatrix@C+30pt{
(A^m,\delta^m) \ar@{-->}[r]^-{\pi^m}_-{!} \ar[d]_{q_A^m}
&(B^m,\epsilon^m) \ar[d]^{q_B^m}
\\
(A,\delta) \ar[r]_-\pi
&(B,\epsilon)
}
\]
commute.
A parallel theory exists for normalizations:
$\pi:(A,\delta)\to (B,\varepsilon)$ is a \emph{normalization} of $(A,\delta)$ if $(B,\varepsilon)$ is normal and $\pi\times G:A\times_\delta G\to B\times_\varepsilon G$ is an isomorphism.
We sometimes call $(B,\varepsilon)$ itself a normalization of $(A,\delta)$.
Normalizations of $(A,\delta)$ always exist, and all are uniquely isomorphic.
Choosing one for every coaction, we get a \emph{normalization functor} that sends $(A,\delta)$ to the normalization
\[
q_A^n:(A,\delta)\to (A^n,\delta^n),
\]
and sends a morphism $\pi:(A,\delta)\to (B,\epsilon)$ to the unique morphism $\pi^n$, called the \emph{normalization} of $\pi$, making the diagram
\[
\xymatrix@C+30pt{
(A,\delta) \ar[r]^-\pi \ar[d]_{q_A^n}
&(B,\epsilon) \ar[d]^{q_B^n}
\\
(A^n,\delta^n) \ar@{-->}[r]_-{\pi^n}^-{!}
&(B^n,\epsilon^n)
}
\]
commute.
Maximalizations and normalizations are automatically surjective.
Moreover, if $\pi:(A,\delta)\to (B,\varepsilon)$ is either a maximalization or a normalization, then $\pi$ maps each \emph{spectral subspace} $A_s:=\{a\in A:\delta(a)=a\otimes s\}$ isometrically onto the corresponding subspace $B_s$, and in particular maps the \emph{fixed-point algebra} $A^\delta:=A_e$ isomorphically onto $B^\epsilon$.
If $(A,\delta)$ is normal, then the maximalization $q_A^m:(A^m,\delta^m)\to (A,\delta)$ is also a normalization of $(A^m,\delta^m)$,
and similarly if $(A,\delta)$ is maximal then the normalization $q_A^n:(A,\delta)\to (A^n,\delta^n)$ is also a maximalization of $(A^n,\delta^n)$.

For every Fell bundle $p:\Aa\to G$, the (full) cross-sectional algebra $C^*(\Aa)$ carries a maximal coaction $\delta_\Aa$, determined on $\Aa$ by $\delta_\Aa(a)=a\otimes p(a)$,
the reduced cross-sectional algebra $C^*_r(\Aa)$ carries a normal coaction $\delta_\Aa^n$ determined by the same formula,
and the \emph{regular representation} $\Lambda_\Aa:(C^*(\Aa),\delta_\Aa)\to (C^*_r(\Aa),\delta_\Aa^n)$ is both a maximalization and a normalization.

For $s\in G$, we write $\Chi_s$ for the characteristic function of $\{s\}$, viewed as an element of $B(G)=C^*(G)^*$.
If $(A,\delta)$ is a coaction, we write
\begin{equation}\label{def:delta_s}
\delta_s=(\id\otimes\Chi_s)\circ\delta,
\end{equation}
which is a projection of norm one from $A$ onto the spectral subspace $A_s$.

\vskip 0.2cm
If $A$ is a $C^*$-algebra and $P$ is a discrete semigroup with identity $e$, a
\emph{product system over $P$ of $C^*$-correspondences over $A$} consists
of a semigroup $X$ equipped with a semigroup homomorphism $d
\colon X \to P$ such that: (1)~$X_p := d^{-1}(p)$ is a
$C^*$-correspondence over $A$ for each $p\in P$; (2)~$X_e =
{_A}A_A$; (3)~the multiplication on $X$ implements isomorphisms
$X_p \otimes_A X_q \cong X_{pq}$ for $p,q \in P \setminus
\{e\}$; and (4)~multiplication implements the right and left
actions of $X_e = A$ on each $X_p$.
For $p \in P$ we let $\phi_p:A\to\Ll(X_p)$ be the homomorphism that implements the left action.
Given $p, q \in P$ with $p \not= e$ there is a homomorphism
$\iota^{pq}_p \colon \Ll(X_p) \to \Ll(X_{pq})$ such that
$\iota^{pq}_p(S)(xy) = (Sx)y$ for all $x \in X_p$, $y \in
X_{q}$ and $S \in \Ll(X_p)$. Upon identifying $\Kk(X_e)=A$, we let
$\iota^q_e \colon \Kk(X_e)\to \Ll(X_{q})$ be given by $\iota^q_e=\phi_q$, see \cite[\S 2.2]{SY}.

Recall that for a $C^*$-correspondence $Y$ over $A$, a map $\psi:Y\to B$ and
a homomorphism $\pi:A\to B$ into a $C^*$-algebra form a \emph{Toeplitz representation} if
$\psi(x\cdot a)=\psi(x)\pi(a)$ and $\pi(\langle x, y\rangle)=\psi(x)^*\psi(y)$ for
all $a\in A$, $x,y\in Y$. A map $\psi$ of a product system $X$ into a $C^*$-algebra $B$ is a \emph{Toeplitz representation} if
$\psi(xy)=\psi(x)\psi(y)$ for all $x,y\in X$ and $(\psi\vert_{X_p},\psi\vert_{X_e})$ is a Toeplitz
representation of the $C^*$-correspondence $X_p$, for all $p\in P$.

\vskip0.2cm
We recall from \cite{N} that a quasi-lattice ordered group $(G, P)$ consists of a subsemigroup
$P$ of a (discrete) group $G$ such that $P\cap P^{-1}=\{e\}$ and every finite subset of $G$ with a
common upper bound in $P$ admits a least common upper bound in $P$, all taken with respect to the left-invariant partial order
on $G$ given by $x\leq y$ if $x^{-1}y\in P$. We write $x\vee y<\infty$ to indicate that $x,y$ have
a common upper bound in $P$, and then $x\vee y$ denotes their least common upper bound in $P$.
If no common upper bound of $x,y$ exists in $P$ we write $x\vee y=\infty$. The semigroup $P$
is \emph{directed} if $x\vee y<\infty$ for all $x,y\in P$.

Given a quasi-lattice ordered group $(G,P)$, a product system
$X$ over $P$ is called \emph{compactly aligned} if $\iota^{p \vee q}_p(S) \iota^{p
\vee q}_q(T) \in \Kk(X_{p\vee q})$ whenever $S \in \Kk(X_p)$
and $T \in \Kk(X_q)$, and $p \vee q < \infty$ cf. \cite{CLSV} or \cite[Definition 5.7]{F}
in case each $X_p$ is essential.
If $\psi:X\to B$ is a Toeplitz representation,  there are
homomorphisms $\psi^{(p)}:\Kk(X_p)\to B$ such that $\psi^{(p)}(\theta_{x,y})=\psi_p(x)\psi_p(y)^*$ for all
$p\in P$ and $x,y\in X$, \cite{P}. When $X$ is compactly aligned, $\psi$ is said to be \emph{Nica covariant} if $\psi^{(p)}(S)\psi^{(q)}(T)$
is $\psi^{(p\vee q)}(\iota^{p \vee q}_p(S) \iota^{p\vee q}_q(T))$ in case $p\vee q<\infty$ and is zero otherwise, see \cite{F}.

Fowler introduced a $C^*$-algebra $\Tc(X)$ and showed it is universal for Nica covariant Toeplitz
representations of $X$, \cite{F}. Here we shall use the notation $\NT(X)$ instead of $\Tc(X)$
because, as advocated for in  \cite[Remark 5.3]{BanHLR}, the choice of
 $\Tc(X)$ for a $C^*$-algebra generated by a
universal representation was unfortunate. Fowler introduced also a Cuntz-Pimsner algebra of $X$
 by imposing usual Cuntz-Pimsner covariance in the sense of \cite{P} in each fibre $X_p$.

 Given a quasi-lattice ordered group $(G,P)$ and a compactly aligned product system $X$
over $P$ of $C^*$-correspondences over $A$, Sims and Yeend \cite{SY} introduced a new notion of Cuntz-Pimsner covariance
 for a Toeplitz representation of $X$. The definition is quite complicated and we will not give it here. It was
proved in \cite[Theorem 4.1]{SY} that the universal Cuntz-Nica-Pimsner covariant representation $j_X$ of $X$ is
injective, meaning that $j_X\vert_{X_e}$ is injective, if $X$ is $\tilde\phi$-injective
(see \cite[\S 2.4]{CLSV} for the definition of this concept). The universal $C^*$-algebra for
Cuntz-Nica-Pimsner covariant representations of $X$, denoted $\NO{X}$, is then nontrivial.

\section{Gauge-invariant uniqueness  for $\NT(X)$ and $\NO{X}$}
\label{NO}

Fix a quasi-lattice ordered group $(G,P)$ and a compactly aligned product system $X$
over $P$ of $C^*$-correspondences over $A$. There is a canonical coaction $(\NT(X),\delta)$ of $G$, and we let
$\mathcal{B}$ be the associated Fell bundle. If $X$ is $\tilde\phi$-injective, there is also a canonical coaction
$(\NO{X},\nu)$ of $G$, whose associated Fell bundle is denoted by $\mathcal{N}$.  It was shown in \cite[Remark 4.5]{CLSV}
that $C^*(\mathcal{B})\cong \NT(X)$ and $C^*(\mathcal{N})\cong \NO{X}$. Equivalently,
 both coactions $\delta$ on $\NT(X)$ and $\nu$ on $\NO{X}$ are maximal in the sense of \cite{EKQ}.

The following terminology was introduced in \cite[Definition 4.10]{CLSV}: $\NO{X}$ has the \emph{gauge-invariant
uniqueness property} provided that a surjective homomorphism $\varphi:\NO{X}\to B$ is injective if
and only if:
\begin{enumerate}\renewcommand{\theenumi}{GI\arabic{enumi}}
\item\label{it:carries coaction} there is a coaction
    $\beta$ of $G$ on $B$ such that
    $\varphi$ is $\nu-\beta$ equivariant,
    and
\item\label{it:inj on A} the homomorphism $\varphi\vert_{j_X(A)}$ is injective.
\end{enumerate}

The \emph{gauge-invariant uniqueness theorem} for $\NO{X}$  is \cite[Corollary 4.11]{CLSV}
and gives a number of necessary
and sufficient conditions for $\NO{X}$  to have the  gauge-invariant uniqueness property. For instance, $\NO{X}$
has the gauge-invariant uniqueness property precisely when the gauge-coaction $\nu$ is normal. Thus  the gauge-invariant
uniqueness theorem holds for $\NO{X}$ provided that $\nu$ is both maximal and normal.

In the next result we recast the gauge-invariant uniqueness property for $\NO{X}$ by asking for a maximal
 coaction on the target algebra. The apparently short proof follows from the general uniqueness
 theorems worked out in the context of Fell bundles in the appendix, and illustrates the power of coaction
 techniques.

 \begin{theorem}[The gauge-invariant uniqueness theorem for $\NO{X}$ and maximal coactions]
 \label{GIUTNO}
 Let $(G,P)$ be a quasi-lattice ordered group
 and $X$ a $\tilde\phi$-injective compactly aligned product system over $P$ of $C^*$-correspondences over $A$. A surjective
homomorphism  $\pi:\NO{X}\to B$ is injective if and only if $\pi$ is injective
on $\NO{X}^\nu$ and there is
 a \emph{maximal} coaction $\beta$ on $B$ such that $\pi$ is $\nu-\beta$ equivariant.
  \end{theorem}

\begin{proof}
Apply Corollary~\ref{cor:giut-max-coact} to $(\NO{X}, \nu)$ and $\pi$.
\end{proof}

Compared to \cite[Corollary 4.12]{CLSV}, Theorem~\ref{GIUTNO} does
not require amenability of $G$, and can be  applied to arbitrary
$\tilde\phi$-injective compactly aligned product systems over $P$ (for which $j_X$ is an
injective representation). The drawback is that $\pi$ needs to be
injective on the entire fixed-point algebra for $\nu$, and not just on the coefficient
algebra.

In practice, the injectivity of $\pi$ on $\NO{X}^\nu$   is likely to be difficult
to establish. However, when the compactly aligned product system satisfies one of the
two conditions: the left actions on the fibres of $X$ are all injective, or $P$ is directed and $X$ is $\tilde\phi$-injective,
then \cite[Theorem 3.8]{CLSV} says that
$\pi$ is injective on $\NO{X}^\nu$  precisely when it is injective as a Toeplitz representation, i.e. its restriction to
$j_X(A)$ is an injective homomorphism.

\begin{example} Suppose that $G$ is a nonabelian finite-type Artin group. Then $G$ and
its positive cone $P$ form a quasi-lattice ordered group. By \cite{CLac1}, $P$ is directed
and $G$ is not amenable. Then, if $X$ is the product system over $P$ with fibers $\CC$,
the algebra $\NO{X}$ is isomorphic to $C^*(G)$ and does not have the gauge-invariant
uniqueness property (see \cite[Remark 5.4]{CLSV} for details).  However, since $\NO{X}^\nu=\CC$,
Theorem~\ref{GIUTNO} implies that a surjective homomorphism $\pi:C^*(G)\to B$ is injective if and only if $B$ carries
a  compatible maximal coaction.
\end{example}

Since also $(\NT(X), \delta)$ is a maximal coaction, we have a version of
\thmref{GIUTNO} for $\NT(X)$.

 \begin{theorem}[The gauge-invariant uniqueness theorem for $\NT(X)$ and maximal coactions]
 Let $(G,P)$ be a quasi-lattice ordered group
 and $X$ a compactly aligned product system over $P$ of $C^*$-correspondences over $A$.
 A surjective homomorphism $\pi:\NT(X)\to B$ is injective if and only if $\pi$ is
injective on $\NT(X)^\delta$ and
 there is a maximal coaction $\beta$ on $B$
 such that $\pi$ is $\delta-\beta$ equivariant.
 \end{theorem}

 \begin{proof}
 Apply Corollary~\ref{cor:giut-max-coact} to $(\NT(X),\delta)$ and $\pi$.
 \end{proof}

\begin{cor}
\label{thm:giut-Tcov}
Let $(G,P)$ be a quasi-lattice ordered group
 and  $X$ a compactly aligned product system over $P$ of $C^*$-correspondences over $A$. The coaction $(\NT(X), \delta)$
 is normal precisely when the following is satisfied: a surjective homomorphism $\pi:\NT(X)\to B$ is injective
 if and only if $\pi$ is injective on $\NT(X)^\delta$ and
 there is a coaction $\beta$ on $B$ such that $\pi$ is $\delta-\beta$ equivariant.
 \end{cor}

\begin{proof} Apply Corollary~\ref{cor:characterise-normal}.
\end{proof}

\begin{cor} Let $(G,P)$ be a quasi-lattice ordered group
 and $X$ a $\tilde\phi$-injective compactly aligned product system over $P$ of $C^*$-correspondences over $A$. The
 coaction $(\NO{X}, \nu)$ is normal precisely when the following is satisfied: a surjective homomorphism $\pi:\NO{X}\to B$ is injective
 if and only if $\pi$ is injective on $\NO{X}^\nu$ and
 there is a coaction $\beta$ on $B$ such that $\pi$ is $\nu-\beta$ equivariant.
\end{cor}

\begin{proof} Apply Corollary~\ref{cor:characterise-normal}.
\end{proof}

Next we recall from \cite{CLSV} that given
 a quasi-lattice ordered group $(G, P)$ and a compactly aligned product system $X$ over $P$ satisfying one of the following two conditions: the
left actions on the fibres of $X$ are all injective, or $P$ is directed and $X$ is $\tilde\phi$-injective, then the
$C^*$-algebra $\NO{X}^\reduced:=C_r^*(\mathcal{N})$ and the normalization $\nu^n$ of $\nu$
have the co-universal property of \cite[Theorem 4.1]{CLSV}. This co-universal property was used to identify various
reduced crossed product type $C^*$-algebras in the form $\NO{X}^\reduced$ for appropriate $X$, and also to investigate the
gauge-invariant uniqueness property in several contexts.

Our abstract uniqueness results for Fell bundles allow us to give a characterization of
injectivity of homomorphisms $\pi:B\to \NO{X}^\reduced$ that is an alternative to \cite[Corollary 4.9]{CLSV}.

\begin{thm}[The gauge-invariant uniqueness theorem for $\NO{X}^\reduced$ and normal coactions]
\label{GIUT-NOXreduced}
 Let
$(G, P)$ be a quasi-lattice ordered group and $X$ a $\tilde\phi$-injective compactly aligned product system over $P$ of
$C^*$-correspondences over $A$. A homomorphism
$\pi:B\to \NO{X}^\reduced$ is injective if and only if there is a normal coaction $\beta$ of $G$ on $B$ such that
$\pi$ is $\beta-\nu^n$ equivariant and $\pi\vert_{B_e}$ is injective.
\end{thm}

\begin{proof} Apply Corollary~\ref{cor:giut-nor-coact}.
\end{proof}

To see how this relates to \cite{CLSV}, suppose $X$ is a compactly aligned product system over $P$ such that the
left actions on the fibres of $X$ are all injective, or $P$ is directed and $X$ is $\tilde\phi$-injective. Suppose also that
$\pi$ arises from the co-universal property of $\NO{X}^\reduced$ applied to an injective Nica covariant
Toeplitz representation $\psi:X\to B$, where there is a coaction $\beta$ of $G$ on $B$ making $\pi$ a $\beta-\nu^n$ equivariant homomorphism.
It is proved in \cite[Corollary 4.9]{CLSV} that $\pi$ is injective if and only if $\beta$ is normal and $\psi$  is Cuntz-Nica-Pimsner covariant.
In Theorem~\ref{GIUT-NOXreduced} the last condition is replaced by  $\pi\vert_{B_e}$ being injective.

\begin{cor}
Let
$(G, P)$ be a quasi-lattice ordered group and $X$ a $\tilde\phi$-injective compactly aligned product system over $P$ of
$C^*$-correspondences over $A$. The coaction $(\NO{X}^\reduced, \nu^n)$ is maximal precisely when the following is satisfied: a
homomorphism $\pi:B\to \NO{X}^\reduced$ is injective if and only if there is a coaction $\beta$ of $G$ on $B$ such that
$\pi$ is $\beta-\nu^n$ equivariant and $\pi\vert_{B_e}$ is injective.
\end{cor}

\begin{proof} Apply Corollary~\ref{cor:characterise-max}.
\end{proof}

\section{Inner coactions}\label{section:inner}

In this section we study inner coactions in relation to faithfulness of representations.
First we recall some notation. The multiplier algebra
$M(C_0(G)\otimes C^*(G))$ is identified with the algebra of continuous bounded functions on $G$ with values in
$M(C^*(G))$ equipped with the
strict topology. Let $w_G$ be the unitary element of $M(C_0(G)\otimes C^*(G))$ given by the canonical embedding of $G$ in $M(C^*(G))$.
 Given a coaction $(A,\delta)$ and a $C^*$-algebra $D$, nondegenerate homomorphisms $\mu:C_0(G)\to M(D)$
and $\pi:A\to M(D)$ form a \emph{covariant pair} for $(A,\delta)$ provided that
the diagram
\[
\xymatrix@C+70pt{
A \ar[r]^-\delta \ar[d]_\pi
&A\otimes C^*(G) \ar[d]^{\pi\otimes\id}
\\
M(D) \ar[r]_-{\ad\bar{\mu\otimes\id}(w_G)\circ(\id\otimes 1)}
&M(D\otimes C^*(G))
}
\]
commutes, or, equivalently,
since $G$ is discrete,
provided that
\begin{equation}\label{def:cov-pair-coaction}
\pi(a_x)\mu(\Chi_y)=\mu(\Chi_{xy})\pi(a_x)
\end{equation}
for all $a_x\in A_x$, and all $x,y\in G$ (see  e.g.,  \cite[Section~2]{EQ}).

By \cite[Lemma~1.11]{QuiggFull}, any nondegenerate homomorphism $\mu:C_0(G)\to M(A)$ implements an \emph{inner} coaction $\delta^\mu$ on $A$ via
\[
\delta^\mu(a)=\ad\bar{\mu\otimes\id}(w_G)(a\otimes 1).
\]
Note that $(\id_A, \mu)$ forms a covariant pair for  every inner coaction $(A,\delta^\mu)$.

Every inner coaction is normal, by \cite[Proposition~2.3]{QuiggFull} (see also \cite[Lemma A.2]{BKQ}).
If $(A, G, \delta^\mu)$ is an inner coaction,
then $a\in A_e$ if and only if $a$ commutes with $\{\mu(\Chi_x): x\in G\}$. Indeed,
if $a\in A_e$ then $a$ commutes with every $\mu(\Chi_x)$ by  \eqref{def:cov-pair-coaction}.
Conversely, if $a$ commutes with every $\mu(\Chi_x)$ then $a$ commutes with $\mu(C_0(G))$, hence $a\otimes 1$ commutes with $\mu(C_0(G))\otimes C^*(G)$, and therefore with $\bar{\mu\otimes\id}(w_G)$, so $a\in A_e$.

\begin{rmk}\label{rmk:inner-via-projections}
We note that a
necessary and sufficient condition for a coaction $\delta$ on $A$ to be inner is that there is a family $\{p_x: x\in G\}$
of orthogonal projections in $M(A)$ that sum strictly to $1$ in $M(A)$  and satisfy
\begin{equation}\label{cond:et}
a_xp_y=p_{xy}a_x
\end{equation}
for all $a_x\in A_x \text{ and }x, y\in G$. Indeed, if $(A, \delta)$ is an inner coaction, there is a nondegenerate
homomorphism $\mu:C_0(G)\to M(A)$ such that $\id_A$ and $\mu$  satisfy \eqref{def:cov-pair-coaction}, which
turns into \eqref{cond:et}  by letting $p_y=\mu(\Chi_y)$.

Conversely, given a coaction $(A, \delta)$ and a family of projections satisfying \eqref{cond:et},
let $\mu:C_0(G)\to M(A)$ be the unique homomorphism satisfying $\mu(\Chi_y)=p_y$ for $y\in G$. Then $\mu$ is nondegenerate because $\sum_{y\in G}p_y=1$ strictly in
$M(A)$, and $(\id_A, \mu)$ forms a covariant pair by \eqref{cond:et}.
Unravelling the definitions, we have $\delta=\delta^\mu$.
\end{rmk}

\begin{theorem}[Abstract uniqueness theorem]\label{thm:abstract-CKUT} Let $(A, \delta)$ be an inner coaction. A
surjective homomorphism $\varphi:A\to B$ onto a $C^*$-algebra $B$ is injective if and only if $\varphi\vert_{A_e}$
is injective.
\end{theorem}

\begin{proof} Since $\delta$ is inner, there is a nondegenerate
homomorphism $\mu:C_0(G)\to M(A)$ such that $\delta=\delta^\mu$. Define a nondegenerate homomorphism $\nu:C_0(G)\to M(B)$ by $\nu=\bar\varphi\circ\mu$.
Then $\delta^\nu$ is an inner coaction on $B$, and the computation
\begin{align*}
\varphi(a_x)\nu(\Chi_y)&=\varphi(a_x)\overline{\varphi}(\mu(\Chi_y))=\varphi(a_x\mu(\Chi_y))\\
&=\varphi(\mu(\Chi_{xy})a_x) \text{ by } \eqref{def:cov-pair-coaction}\\
&=\nu(\Chi_{xy})\varphi(a_x)
\end{align*}
for $a_x\in A_x$ and $x,y\in G$ shows that $\varphi$ is $\delta-\delta^\nu$ equivariant. The theorem therefore follows from Proposition~\ref{GIUTFB} (2).
\end{proof}

Suppose that $(A, \delta)$ is a coaction of $G$. An ideal $I$ in $A$ is \emph{$\delta$-invariant}
if the restriction of $\delta$ to $I$ gives rise to a coaction of $G$ on $I$. If this is the case,
we let $\delta\vert_I$ be the restricted coaction on $I$.

\begin{cor}[Essential-inner uniqueness theorem]\label{cor:eiut} Let $(A, \delta)$ be a coaction of $G$ and $I$
 a $\delta$-invariant ideal in $A$ such that the coaction $\delta\vert_I$ of $G$ on $I$
is inner. If $I$ is an essential ideal in $A$, then
a homomorphism $\varphi:A\to B$ is injective if and only if $\varphi\vert_{A^\delta}$ is injective.
\end{cor}

\begin{proof} For the non-trivial direction, suppose that $\varphi\vert_{A^\delta}$ is injective.
Then $\varphi$ is injective on $I^{\delta\vert_I}=A^\delta\cap I$. Since $\delta\vert_I$ is inner,
Theorem~\ref{thm:abstract-CKUT} implies that $\varphi\vert_I$ is injective. But $I$
is an essential ideal, and so $\varphi$ is injective.
\end{proof}

\section{$C^*$-algebras of quasi-lattice ordered groups}\label{subsect:Nica-alg-semigp-cp}

In this section we recall Nica's constructions of $C^*$-algebras associated
to isometric representations of quasi-lattice ordered groups, we give a quick review of subsequent constructions, and
we make connections with coaction theory.

Let $(G, P)$ be a quasi-lattice ordered group. A semigroup homomorphism $V$ of $P$ into the isometries on a Hilbert
space $H$ such that $V_e=I$ and $V_sV_t=V_{st}$ for all $s,t\in P$ is called an (isometric) representation of $P$.
Let $\{\varepsilon_t\}_{t\in P}$ be the canonical orthonormal basis of $l^2(P)$.
The \emph{Toeplitz} or \emph{Wiener-Hopf representation} of $P$ on $l^2(P)$ is
given by $T_s\varepsilon_t=\varepsilon_{st}$, for $s,t\in P$. The \emph{Toeplitz algebra}
(or Wiener-Hopf algebra)  $\Tt(G, P)$ is the $C^*$-subalgebra of $B(l^2(P))$ generated by the image of $T$. Nica
noticed that $T_sT_s^*T_tT_t^*=T_{s\vee t}T_{s\vee t}^*$ when $s\vee t<\infty$ and is zero otherwise. Such
representations of $P$ are now called Nica covariant, and $C^*(G,P)$ is the universal $C^*$-algebra
generated by a Nica covariant representation  $v$ of $P$ (see \cite{N, LacR1}).

By \cite[Proposition 3.2]{N}, the family $\{T_sT_t^*: s,t \in P\}$ spans a dense subalgebra of $\Tt(G, P)$.
The \emph{diagonal subalgebra} of $\Tt(G, P)$ is $\Dd=\clsp\{T_sT_s^*: s\in P\}$.

We next recall some facts from \cite{LacR1}. Let $(G, P)$ be a quasi-lattice ordered group, and for each $s\in P$
write $1_s$ for the characteristic function  of the set $\{t\in P:s\leq t\}$.
Then  $B_P=\clsp\{1_s: s\in P\}$ is a commutative $C^*$-subalgebra of
$l^\infty(P)$, and $C^*(G,P)$ is the semigroup
crossed product $B_P\rtimes P$ arising from translation $t\mapsto (1_s\to 1_{ts})$ on $B_P$, see
\cite[Corollary 2.4]{LacR1}. By \cite[\S 6.1]{LacR1}, there is a coaction
$\delta$ of $G$ on $C^*(G,P)$ such that $\delta(v_s)=v_s\otimes s$ for all $s\in P$, and  $B_P$ is
the fixed-point algebra  $C^*(G,P)^\delta$.  Moreover,
\cite[Proposition 2.3]{LacR1} shows that every representation of $C^*(G,P)$ is determined by a
Nica covariant representation of $P$. We let ${\lambda_T}$ denote the representation of $C^*(G, P)$
determined by $T$, and note that it carries $1_s$ to $T_sT_s^*$ for all $s\in P$.

It follows from \cite[Proposition 5.6]{SY} that if  $X=\CC\times P$ is the
trivial product system over $P$ with fibers $X_p={}_{\CC}\CC_{\CC}$ for all $p\in P$, then $\NT(X)\cong C^*(G,P)$.
Since $\delta$  is maximal by \cite[Remark 4.5]{CLSV}, we shall view it
as a coaction on $C^*(G, P)=C^*(\mathcal{B})$ (recall that we let $\mathcal B$ denote the associated Fell bundle over $G$) with fixed point algebra equal to $B_P$. Recall
from \cite{EQ} that $C^*_r(\mathcal{B})$ is identified with the normalization $(C^*(\mathcal{B}))^n.$

\begin{prop}\label{prop:normalising-delta}
The representation $\lambda_T$ is both a maximalization and a normalization from
$(C^*(G, P), \delta)$ onto $(\Tt(G, P), \delta^n)$. In particular,
$\Tt(G, P)\cong C^*_r(\mathcal{B})$.
\end{prop}

\begin{proof}
Using reduced coactions, it was shown in \cite[Proposition 6.5]{QuiggRa} that there is a
normal coaction $\eta$  on $\Tt(G, P)$ such that $\eta(T_sT_t^*)=T_sT_t^*\otimes st^{-1}$ for $s,t \in P$.
Then $\lambda_T: (C^*(G, P), \delta)\to (\Tt(G, P), \eta)$ is equivariant.

We noted in the preliminaries that the regular representation
$\Lambda_{\Bb}:(C^*(G, P), \delta)\to (C^*_r(\mathcal{B}), \delta^n)$ is both a
maximalization and a normalization. Since $\lambda_T$
is injective on $B_P$ by \cite[Corollary 2.4(1)]{LacR1}, Proposition~\ref{GIUTFB}, parts (\ref{it:delta-max}) and (\ref{it:epsilon-normal}),
imply that $\lambda_T$ is also both a maximalization and a normalization.
Since all maximalizations are isomorphic, and similarly for normalizations, we therefore have $C^*(G, P)\cong (\Tt(G, P))^m$ and
 $\Tt(G, P)\cong (C^*(G, P))^n$.
\end{proof}

Nica \cite[Definition 4.2]{N} defined $(G, P)$ to be amenable if  the representation $\lambda_T$
is an isomorphism. His definition motivated
Exel's definition of amenable Fell bundles in \cite{E}. Our
Proposition~\ref{prop:normalising-delta} shows that the Fell bundle $\mathcal{B}$ is
amenable when   $(G, P)$ is amenable in Nica's sense.

\section{Finite exhaustive sets of strictly positive elements}\label{section:fessub}

Throughout this section let $(G,P)$ be a quasi-lattice ordered group.

\begin{defn}
A \fp\ of $(G,P)$ is a finite subset $F\subset P\minus\{e\}$ such that $FP=P\minus\{e\}$.
\end{defn}

``\fp'' stands for ``finite exhaustive set of strictly positive elements'', and the existence of such an $F$ is easily seen to be equivalent to the existence, for each $x\in G$, of a finite set of strict upper bounds $S$ of $x$ (i.e., $x\lneqq y$ for all $y\in S$) that is exhaustive in the sense that every strict upper bound of $x$
has a lower bound in $S$ --- namely, take $S=xF$.
This condition was introduced
in \cite{N}, and was shown in
\cite[Proposition 6.3]{N} to be equivalent, among others, to the fact that $\Tt(G,P)$
contains the compact operators $\Kk(l^2(P))$.

As remarked in \cite{N}, all pairs $(G,P)$ with $P$ finitely generated have a \fp. In particular,
$(\FF_n, \FF_n^+)$ has a \fp\ for all $n\geq 1$. The pair $(\FF_\infty, \FF_\infty^+)$
does not have a \fp\ since in this case the Toeplitz algebra is isomorphic to $\Oo_\infty$ and is therefore simple.
Another example of
a quasi-lattice ordered group not having a \fp\ is $(\QQ_+^*, \NN^\times)$, endowed with the order given by
$r\leq s \iff r \text{ divides }s$. No finite set of non-zero positive integers different from $1$
can contain a lower bound for every element in $\NN^\times\setminus\{1\}$.

\begin{ex}
It is possible for $(G,P)$ to have a \fp\ but not be finitely generated.
For example,
consider $G = (R,+)$ and
$P={0}\cup[1,\infty)$.  Then $(G,P)$ is quasi-lattice ordered and has a \fp\ (and $P-P=G$), but
is not finitely generated.
\end{ex}

The following result is the essential-inner uniqueness theorem for $\Tt(G, P)$ when $(G, P)$ has a \fp.

\begin{theorem}\label{thm:describe-inner-Nica-alg} Let $(G, P)$ be a quasi-lattice ordered group and  $\delta^\normal$
 the canonical normal coaction on $\Tt(G, P)$. Assume $(G, P)$ has a \fp. Then the following assertions hold.

 \textnormal{(a)} $\Kk(l^2(P))$ is a $\delta^\normal$-invariant ideal in $\Tt(G, P)$ and $\delta^\normal\vert_{\Kk(l^2(P))}$ is an inner coaction.

 \textnormal{(b)} Let $\varphi$ be a homomorphism of $\Tt(G,P)$ into a $C^*$-algebra $B$.
Then the following are equivalent:
\begin{enumerate}
\item\label{it:phi-inj} The homomorphism $\varphi$ is injective.
\item\label{it:phi-inj-fix} The homomorphism $\varphi$ is injective on $\Dd$.
\item\label{it:phi-inj-chie} We have $\varphi(p_e)\not=0$, where $p_e$ is the rank-one projection onto $\varepsilon_e$.
\end{enumerate}
\end{theorem}

To prove this theorem we shall need some preparation. The equivalence of (1) and (4) in the next result is implicit in \cite[Proposition 6.3]{N}.
We first recall a couple of facts about the Nica spectrum of $(G, P)$.

The spectrum of the commutative algebra $\Dd$  is the space $\Omega$ of all non-empty, hereditary, directed
subsets $A$ of $P$, see \cite[\S6]{N} for definitions and details. Assigning the
set $A_\gamma=\{s\in P: \gamma(T_sT_s^*)=1\}$ to a character $\gamma$ of $\Dd$ gives a homeomorphism of the character space of
$\Dd$ onto $\Omega$. Let
$\iota:P\to \Omega$ be the map $t\mapsto [e,t]$ from\cite[\S 6.3, Remark 1]{N}, where $[e,t]:=\{s\in P: s\leq t\}$.
Since $\lambda_T$ is an isomorphism of $B_P$ onto $\Dd$, there is a homeomorphism
$\widehat{B_P}\to \Omega$ given by $\gamma\to A_\gamma$ for
$A_\gamma=\{t\in P: \gamma(1_t)=1\}$, see \cite{Lac1}. Under this homeomorphism, $[e,t]$ corresponds to the character $\gamma$
of $B_P$  given by $\gamma(1_x)=1_x(t)$ for all $x\in P$.

\begin{lemma}\label{lem:essential-ideal}
Let $(G, P)$ be a quasi-lattice ordered group. The following statements are equivalent:
\begin{enumerate}
\item\label{it:pair-to-ideal1} $(G, P)$ has a \fp.
\item\label{it:pair-to-ideal2} $c_0(P)$ is contained in $B_P.$
\item\label{it:pair-to-ideal3} $c_0(P)$ is an essential ideal in $B_P$.
\item\label{it:pair-to-ideal4} $c_0(\iota(P))$ is an essential ideal in $\Dd$.
\end{enumerate}
\end{lemma}

\begin{proof}
Let $F$ be a \fp\ for $(G,P)$, and for $x\in P$ define
$1_{\{x\}}\in l^\infty(P)$ by
\[
1_{\{x\}}(y)=\begin{cases}
1\case y=x\\
0\ifnot.
\end{cases}
\]
Then $c_0(P)$ is generated by the projections $1_{\{x\}}$ for $x\in P$.
To establish \eqref{it:pair-to-ideal1}$\Rightarrow$\eqref{it:pair-to-ideal2} it suffices to prove that
\begin{equation}\label{eq:Chi-x}
\prod_{a\in F}(1_x-1_{xa})=1_{\{x\}}
\end{equation}
 for all $x\in P$. Take $y\in P$, and note that if $x=y$ then
$1_{xa}(x)=0$ for all $a\in F$, so the left hand side of \eqref{eq:Chi-x} evaluated at $y$
is equal to $1_x(x)=1$.

If  $x^{-1}y\in P\setminus \{e\}$, there is $a\in F$ such that $a\leq x^{-1}y$. Hence
$(1_x-1_{xa})(y)=1_x(y)-1_{xa}(y)=0$, and so the  product on the left hand side of \eqref{eq:Chi-x} evaluated at $y$ is zero.

The remaining possibility for $y$ is $x\inv y\notin P$, in which case the left side \eqref{eq:Chi-x} is obviously zero; this establishes \eqref{it:pair-to-ideal1}$\Rightarrow$\eqref{it:pair-to-ideal2}.

The implication \eqref{it:pair-to-ideal2}$\Rightarrow$\eqref{it:pair-to-ideal3} is clear, and
\eqref{it:pair-to-ideal3}$\Rightarrow$\eqref{it:pair-to-ideal4} follows because the isomorphism $B_P\to \Dd$ carries
the ideal $c_0(P)$ onto  $c_0(\iota(P))$.

It remains to prove \eqref{it:pair-to-ideal4}$\Rightarrow$\eqref{it:pair-to-ideal1}. Since $c_0(\iota(P))$ is essential,
$\iota(P)$ is an open and dense subset of $\Omega$. Thus
the relative topology on $\iota(P)$ is the original topology on $P$, and so each interval
$[e,t]$ is open in $\Omega$. Then the implication $2\Rightarrow 4$ from \cite[Proposition 6.3]{N}
shows that $(G, P)$ has a \fp.
\end{proof}

The next result is a  sharpening of \cite[remark 2.2.3]{N}.

\begin{lemma}\label{lem:lub-products}
Let $(G, P)$ be a quasi-lattice ordered group. Assume $a,b,z\in G$ are such that at least one of $a,b$
is in $P$ and at least one of $za, zb$ is in $P$. Then $a\vee b<\infty$ precisely when
$z(a\vee b)<\infty$, in which case $z(a\vee b)=za\vee zb$ as
elements of $P$.
\end{lemma}

\begin{proof}
Suppose $w:=za\vee zb\in P$. Then $a\leq z^{-1}w$ and $b\leq z^{-1}w$. Since
at least one of $a,b$ is in $P$ we necessarily have $z^{-1}w\in P$. Thus $a\vee b< \infty$. Since the order is left-invariant, $w\leq z(a\vee b)$. Then $z^{-1}w\leq a\vee b$,
so necessarily $z^{-1}w=a\vee b$.

Now suppose $a\vee b<\infty$. Then
by left invariance $za\leq z(a\vee b)$ and $zb\leq z(a\vee b)$. It follows that
$z(a\vee b)\in P$. Therefore  $za\vee zb\leq z(a\vee b)$, from which equality follows as in the previous paragraph.
\end{proof}

\begin{lem}\label{prop:spectral-stuff}
Let $(G, P)$ be a quasi-lattice ordered group having a \fp. The assignment
\begin{equation}\label{eq:et}
e_y=\begin{cases}
1_{\{y\}}&\text{ if }y\in P\\
0&\text{ if }y\in G\setminus P
\end{cases}
\end{equation}
for $y\in G$ defines a family of mutually orthogonal projections in $C^*(G, P)$ such that
\begin{equation}\label{need-delta-inner}
c_xe_y=e_{xy}c_x
\end{equation}
for all $c_x\in C^*(G, P)_x$, and $x, y\in G$.
\end{lem}

Note that the family $\{e_y\}$ gives rise to a nondegenerate homomorphism
$\mu:c_0(G)\to c_0(P)$.

\begin{proof} Since $C^*(G, P)$ is the closed span of monomials $v_pv_q^*$, we have
\begin{equation}\label{eq:spectral-subsp-full}
C^*(G, P)_x=\begin{cases}
\clsp\{v_pv_q^*: x=pq^{-1}\}&\text{ if }x\in PP^{-1}\\
0&\text{ otherwise}.
\end{cases}
\end{equation}
It therefore suffices to prove
\eqref{need-delta-inner} when $c_x$ is of form $v_pv_q^*$ with $x=pq^{-1}\in PP^{-1}.$

\smallskip
\noindent{{\bf Case 1}: $y, xy\in P$.} We must show that $v_pv_q^*e_y=e_{xy}v_pv_q^*$. By
\eqref{eq:Chi-x} we have $e_y=\prod_{a\in F}(1_y-1_{ya})$ and
$e_{xy}=\prod_{a\in F}(1_{xy}-1_{xya})$.

If $y\vee q=\infty$ then Lemma~\ref{lem:lub-products} implies $xy\vee p=\infty$. Hence
Nica covariance of $v$ implies that $v_q^*1_y=v_q^*v_yv_y^*=0$
and $1_{xy}v_p=v_{xy}v_{xy}^*v_p=0$. Since  also $ya\vee q=\infty$ for all $a\in F$ (because
$y\leq ya$ for all $a\in F$), we likewise have $xya\vee p=\infty$, and therefore
$c_x1_{ya}=1_{xya}c_x$. In all, \eqref{need-delta-inner} is satisfied.

If $y\vee q<\infty$ then $xy\vee p=x(q\vee y)<\infty$ by Lemma~\ref{lem:lub-products}.
We then have
\begin{align}
c_x1_y&=v_pv_q^*1_y
=v_pv_q^*v_yv_y^*\notag \\
&=v_pv_{q^{-1}(q\vee y)}v^*_{y^{-1}(q\vee y)}v_y^*\notag\\
&=v_{x(q\vee y)}v^*_{q\vee y}\notag\\
&=v_{xy\vee p}v^*_{q\vee y}\label{eq:finiteqy}\\
&=v_{xy}v_{(xy)^{-1}(xy\vee p)}v _{p^{-1}(xy\vee p)}^*v^*_q\notag\\
&=1_{xy}v_pv_q^*=1_{xy}c_x.\label{computations-y}
\end{align}
Let $a\in F$. Two sub-cases arise: $ya\vee q=\infty$, in which case also $xya\vee p=\infty$, and
$c_x1_{ya}=1_{xya}c_x=0$ follows as in the previous paragraph. The second sub-case
has $ya\vee q<\infty$, which entails $xya\vee p<\infty$, and replacing $y$ with $ya$ in the computations leading to \eqref{computations-y} shows that $c_x1_{ya}=1_{xya}c_x$. The
equality \eqref{need-delta-inner} is thus satisfied in case 1.

\smallskip
\noindent{{\bf Case 2}: $y\in P$, $xy\notin P$.} We must show $v_pv_q^*e_y=0$. Equivalently,
we must show
\begin{equation}\label{eq:cxey-0}
\prod_{a\in F}v_pv_q^*(1_y-1_{ya})=0.
\end{equation}
Again, two sub-cases arise. If $q\vee y=\infty$, then also $q\vee ya=\infty$ for all
$a\in F$, and by Nica covariance we see that $v_q^*v_y=0=v_q^*v_{ya}$ for all $a\in F$.
Hence \eqref{eq:cxey-0} follows. In case $q\vee y\in P$, we have $v_pv_q^*1_y=
v_{xy\vee p}v^*_{q\vee y}$ by \eqref{eq:finiteqy} (where the use of
Lemma~\ref{lem:lub-products} is legitimate because $p,q,y\in P$).

To establish \eqref{eq:cxey-0} we claim that there exists $a'\in F$ such that
$v_pv_q^*(1_y-1_{ya'})=0$.  The assumption $xy\notin P$ implies $q^{-1}y\notin P$,
and so $y^{-1}(q\vee y)\in P\setminus \{e\}$. Thus by assumption there is $a'\in F$ with
$a'\leq y^{-1}(q\vee y)$. This gives $ya'\vee q\leq q\vee y$, and since the reverse inequality is satisfied because $F\subset P$ we get $q\vee ya'= q\vee y\in P$. Applying
Lemma~\ref{lem:lub-products} yields $xya'\vee p=x(q\vee ya')=x(q\vee y)=xy\vee p$,
and invoking equation~\eqref{eq:finiteqy} where $y$ is replaced by $ya'$ gives
$v_pv_q^*1_{ya'}=v_{xya'\vee p}v^*_{q\vee ya'}$. The claim is therefore proved, and case 2
is finished.

\smallskip
\noindent{{\bf Case 3}: $y\notin P$, $xy\in P$.} We must show that $e_{xy}v_pv_ q^*=0$.
Either $xy\vee p=\infty$, in which case $xya\vee p=\infty$ for all $a\in F$, and
\eqref{eq:cxey-0} follows by Nica covariance, or $xy\vee p\in P$. If this
last alternative happens, the
choice of $y$ implies that $p\nleqq xy$,  so $(xy)^{-1}(xy\vee p)\in P\setminus \{e\}$.
The \fp\ $F$ supplies $a'\in F$ with $a'\leq (xy)^{-1}(xy\vee p)$, and similarly
to case 2 we get $1_{xy}v_pv_q^*=1_{xya'}v_pv_q^*$, from which \eqref{eq:cxey-0} again follows.

\smallskip
\noindent{{\bf Case 4}: $y\notin P, xy\notin P$.} Then both sides of
\eqref{need-delta-inner}  are zero.
\end{proof}

\begin{proof}[Proof of Theorem~\ref{thm:describe-inner-Nica-alg}]
Since $(G, P)$ has a \fp, \cite[Proposition 6.3]{N} gives an inclusion $\Kk(l^2(P))\subset \Tt(G, P)$. Clearly
$\Kk(l ^2(P))$ is an ideal in $\Tt(G, P)$.

To prove part (a) we will show that $\Kk(l ^2(P))$ is $\delta^\normal$-invariant. This will give a
coaction $\delta^\normal_{\Kk(l^2(P))}$ on $\Kk(l ^2(P))$ obtained as restriction of $\delta^\normal$.
We shall then construct mutually orthogonal projections
$\{p_x: x\in G\}$ in $B(l^2(P))$ such that $\sum_{x\in G}p_x=I$ in weak-operator topology on
$B(l^2(P))$ and $a_xp_y=p_{xy}a_x$ for all
$a_x\in \Kk(l ^2(G, P))_x$ and $x,y\in G$. Remark~\ref{rmk:inner-via-projections} therefore provides an inner
coaction $\delta^\mu$ on ${\Kk(l^2(P))}$ with $\delta^\normal_{\Kk(l^2(P))}=\delta^\mu$.

Let $F$ be a \fp\ of $(G,P)$.
For every $x\in G$ define a projection in $B(l^2(P))$ by
\[
p_x=\begin{cases}
\lambda_T(1_{\{x\}})\case x\in P\\
0\case x\in G\minus P.
\end{cases}
\]
Note that for $x\in P$ we have
\[
p_x=\prod_{a\in F}(T_xT_x^*-T_{xa}T_{xa}^*).
\]
With $\xi\otimes \eta$
denoting the rank-one operator $(\xi\otimes \eta)(\zeta)=\eta\langle \xi, \zeta\rangle$ in $B(l ^2(P))$ we see that
$p_e=\varepsilon_e\otimes \varepsilon_e$. So $p_e\in \Kk(l^2(P))$.
Since also $p_e\in \Dd$, it follows that $\delta^\normal(p_e)=p_e\otimes 1$.
For  $x,y\in P$, the product $T_xp_eT_y^*$ is the rank-one operator $e_y\otimes e_x$, see also the proof of
\cite[Proposition 6.3]{N}. Thus $\Kk(l^2(P))$
is the closed span of monomials $T_xp_eT_y^*$ for $x,y\in P$. But
$$
\delta^\normal(T_xp_eT_y^*)=\delta^\normal(T_x)\delta^\normal(p_e)\delta^\normal(T_y^*)=(T_xp_eT_y^*)\otimes xy^{-1},
$$
showing that $\delta^\normal(\Kk(l^2(P)))\subset \Kk(l^2(P))\otimes C^*(G)$. Since
\[
(T_xp_eT_y^*)\otimes z=(T_xp_eT_y^*\otimes xy^{-1})(1\otimes yx^{-1}z),
\]
we have
\[
\clsp\delta^\normal(\Kk(l ^2(P)))(1\otimes C^*(G))=\Kk(l^2(P))\otimes C^*(G).
\]
Thus ${\Kk(l^2(P))}$  is $\delta^\normal$-invariant, and
by restriction $\delta^\normal$ gives a coaction $\delta^\normal\vert_{\Kk(l^2(P))}$ on $\Kk(l ^2(P))$.

Since $\lambda_T$ is a maximalization by Proposition~\ref{prop:normalising-delta}, it carries $C^*(G, P)_x$
isometrically onto $\Tt(G, P)_x$ for every $x\in G$. \lemref{prop:spectral-stuff} implies that
$a_xp_y=p_{xy}a_x$ for every $x\in PP^{-1}$, $a_x\in \Tt(G, P)_x$ and $y\in G$. In particular,
we may take $a_x\in \Kk(l^2(P))\cap \Tt(G, P)_x$, which shows that
$$
\delta^\mu\vert_{\Kk(l^2(P))_x}=\delta^\normal\vert_{\Kk(l^2(P))\cap \Tt(G, P)_x}.
$$
Hence $\delta^\normal\vert_{\Kk(l^2(P))}$ coincides with $\delta^\mu$, and thus is an inner coaction, as claimed in (a).

For (b), obviously it suffices to show \eqref{it:phi-inj-fix}$\Rightarrow$\eqref{it:phi-inj} and \eqref{it:phi-inj-chie}$\Rightarrow$\eqref{it:phi-inj-fix}.
The implication \eqref{it:phi-inj-fix}$\Rightarrow$\eqref{it:phi-inj} follows from (a) and Corollary~\ref{cor:eiut} because $\Kk(l^2(P))$ is an essential ideal in $B(l^2(P))$, hence in
$\Tt(G, P)$.

For the implication \eqref{it:phi-inj-chie}$\Rightarrow$\eqref{it:phi-inj-fix},
suppose $\varphi(p_e)\not=0$.
For every $x\in P$
we have $T_xp_e=p_xT_x$,
and $T_x$ is an isometry,
so
$\varphi(T_x)\varphi(p_e)\not=0$.
Then $\varphi(p_x)\varphi(T_x)\not=0$,
and hence $\varphi(p_x)\not=0$. By linearity and density it follows that $\varphi(M_f)\not=0$ for all $f\in c_0(\iota(P))\subset \Dd$.
Since $c_0(\iota(P))$ is an essential ideal in $\Dd$, it follows that $\varphi$ is injective on $\Dd$.
\end{proof}

\begin{rem}
In the above proof of Theorem~6.3, we appealed to Corollary~4.3 for the
implication (2) $\Rightarrow$ (1), and then for (3) $\Rightarrow$ (2) we
employed an elementary argument. In fact, however, in this particular case we
can prove (3) $\Rightarrow$ (1) directly, as follows. We have
mentioned that FESSPE guarantees $\Kk(l^2(P))\subset \Tt(G,P)$,
and then (3) implies that $\varphi$ is nonzero
on the simple, essential ideal $\Kk(l^2(P))$, and hence is faithful.
Nevertheless, we wanted to show how the method involving Corollary~4.3 can be
applied, because we feel that it will be useful more generally.
\end{rem}

A result similar to Theorem~\ref{thm:describe-inner-Nica-alg} (a) can be proved for $C^*(G, P)$, as follows.

\begin{cor}\label{cor:6.8}
In the notation of Lemma~\ref{prop:spectral-stuff}, the set
\begin{equation}\label{eq:ideal-of-full-algebra}
\mathcal{J}:=\clsp\{v_xe_ev_y^*: x,y\in P\}
\end{equation}
is a $\delta$-invariant ideal of $C^*(G, P)$, and $\delta\vert_{\mathcal{J}}$ is an inner coaction.
\end{cor}

\begin{proof} Since $C^*(G, P)=\clsp\{v_sv_t^*: s,t\in P\}$, it suffices to show that $\mathcal{J}$ is a subalgebra of $C^*(G, P)$ and that $v_sv_t^*v_xe_ev_y^*$ and $v_xe_ev_y^*v_sv_t^*$
are in $\mathcal{J}$ for all $s,t,x,y\in P$. By \eqref{need-delta-inner}, $v_t^*v_xe_e=v_t^*e_xv_x=0$ unless $t^{-1}x\in P$, in which
case $v_sv_t^*v_xe_ev_y^*=v_sv_{t^{-1}x}e_ev_y^*\in \mathcal{J}$. If $y\vee s=\infty$, Nica covariance of $v$ implies that $v_y^*v_s=0$. Otherwise $v_y^*v_s=v_{y^{-1}(y\vee s)}v_{s^{-1}(y\vee s)}^*$, and  \eqref{need-delta-inner} implies that $e_ev_{y^{-1}(y\vee s)}=0$ unless
$y^{-1}(y\vee s)=e$. If $y\vee s=y$ it follows that $v_xe_ev_y^*v_sv_t^*=v_xe_ev_ {t(s^{-1}y)}^*\in \mathcal{J}$. Now clearly
$\mathcal{J}$ is closed under taking adjoints, and by the previous computations it follows that $v_xe_ev_y^*v_se_ev_t^*$ is
zero unless $s=y$, in which case it equals $v_xe_ev_t^*$, so it lies in $\mathcal{J}$.

That $\mathcal{J}$ is $\delta$-invariant follows as in the proof of part (a) of Theorem~\ref{thm:describe-inner-Nica-alg}  because
 $\delta(v_xe_ev_y^*)=v_xe_ev_y^*\otimes xy^{-1}$ and $(v_xe_ev_y^*)\otimes z=(v_xe_ev_y^*\otimes xy^{-1})(1\otimes yx^{-1}z)$. Hence Lemma~\ref{prop:spectral-stuff} implies that $\delta\vert_{\mathcal{J}}$ is $\delta^\mu$, and therefore is an inner coaction.
\end{proof}

We obtain an essential-inner uniqueness theorem for $C^*(G, P)$ if
the ideal
$\mathcal{J}$
of \eqref{eq:ideal-of-full-algebra}
is  essential.

\begin{cor}
If the ideal
$\mathcal{J}$
of \eqref{eq:ideal-of-full-algebra}
is  essential in $C^*(G, P)$,
then
$C^*(G, P)$ has the essential-inner uniqueness property of \corref{cor:eiut}.
This holds in particular if $(G, P)$ has the approximation property for positive definite functions in the sense of Nica.
\end{cor}

\begin{proof}
If $\mathcal{J}$ is essential, then the conclusion follows immediately from \corref{cor:6.8} and \corref{cor:eiut}.
For the other part, note that,
as remarked in [Lac99], if  $(G, P)$ has the approximation property of Nica then for every ideal $\mathcal{I}$ of $C^*(G, P)$
we have $\mathcal{I}=\{X\in C^*(G, P): \Phi(X^*X)\in \Phi(\mathcal{I})\}$, where $\Phi$ is the conditional expectation  from $C^*(G, P)$ into $B_P$, see \cite[Corollaries 2.4 and 3.3]{LacR1}. Now $\Phi$ is faithful by \cite[\S 4.3 and 4.5]{N}.
Therefore, if $\mathcal{I}$ is non-trivial then by faithfulness of $\Phi$ also $\Phi(\mathcal{I})$ is non-trivial as an ideal of $B_P$. By Lemma~\ref{lem:essential-ideal} there exists  $e_x\in c_0(P)$
such that $e_x\in \Phi(\mathcal{I})$. It follows that  $e_x\in\mathcal{I}$, so $\mathcal{I}\cap \mathcal{J}$ is non-trivial, and hence $\mathcal{J}$ is essential.
\end{proof}

The next result is a converse to Theorem~\ref{thm:describe-inner-Nica-alg}.

\begin{theorem}\label{thm:inner-to-FES} Suppose there is a family $\{q_x: x\in G\}\subset \Tt(G, P)$ of mutually orthogonal
projections such that
\begin{enumerate}
\item $q_y\in \Dd'$ and $T_pq_y=q_{py}T_p$ for all $y\in G$ and $p\in P$, and
\item $\sum_{y\in G}q_y=I$ in the weak operator topology of $B(l^2(P))$.
\end{enumerate}
Then $(G, P)$ has a \fp.
\end{theorem}

To prove this theorem we will need a lemma.

\begin{lemma}\label{lem:Toe-mult-free}
Let $(G, P)$ be a quasi-lattice ordered group and $\Dd$ the diagonal
subalgebra of $\Tt(G, P)$. Then the commutant $\Dd'$ is contained in $l^\infty(P)$.
\end{lemma}

\begin{proof}
Let $M\in \Dd'\subset B(l^2(P))$.  We claim that there is $g\in l^\infty(P)$ such that $M=M_g$.
For each $p\in P$ define $f_p:=M\varepsilon_p$ in $l^2(P)$. Using that $MM_{1_p}=M_{1_p}M$
implies that $f_p=M_{1_p}f_p$, and therefore $f_p$ has support included in $\{t\in P: p\leq t\}$.
On the other hand, if $p\leq t$ and $p\neq t$, then the commutation relation $MM_{1_t}=M_{1_t}M$
implies that $M_{1_t}f_p=0$, showing that $f_p$ has support the single point $\{p\}$. Thus there is
$g:P\to \CC$ such that $M\varepsilon_p=g(p)\varepsilon_p$ for all $p\in P$. Since $\Vert M\varepsilon_p\Vert_2\leq \Vert M\Vert$
for all $p\in P$, it follows that $\vert g(p)\vert\leq \Vert M\Vert$ for all $p\in P$. This means $g\in l^\infty$. The claim,
hence the lemma, are proved.
\end{proof}

\begin{proof}[Proof of Theorem~\ref{thm:inner-to-FES}.]
We  will
show that $\Tt(G, P)\supseteq \Kk(l^2(P))$, and then apply \cite[Proposition 6.3]{N} to
conclude that $(G, P)$ has a \fp.

It suffices to show that $\Tt(G, P)$ contains all rank one projections $\varepsilon_y\otimes
\varepsilon_x$  on $l^2(P)$, where $x,y\in P$.  For this it suffices to establish that $q_e=p_e$,
because then we will have $\varepsilon_y\otimes\varepsilon_x=T_xq_eT_y^*$ as in the proof of Theorem~\ref{thm:describe-inner-Nica-alg}.

By assumption (1), $T_pT_p^*q_y=q_yT_pT_p^*$ for all $p\in P$ and  $y\in G$. Thus  $q_y\in\Dd'$ for all $y\in G$, and so
$q_y\in l^\infty(P)$ by Lemma~\ref{lem:Toe-mult-free}. Write $q_y=M_{\Chi_{E(y)}}$  where $\emptyset\neq E(y)\subset P$ for every $y\in G$.
Since $\sum_{y\in G}q_y=I$, the family $\{E(y)\}_{y\in G}$ is a mutually disjoint family such that $P=\cup_{y\in P}E(y)$. We claim that
\begin{equation}\label{eq:Et-single}
E(y)=\{y\} \text{ for all }y\in P.
\end{equation}
Towards the claim, we prove first that $pE(y)=E(py)$ for all $p, y\in P$. By assumption (1), $T_pM_{\Chi_{E(y)}}=M_{\Chi_{E(py)}}T_p$
for $p\in P$. Applying both sides to $\varepsilon_u$ gives
$$
\begin{cases}\varepsilon_{pu}&\text{ if }u\in E(y)\\
0&\text{ if }u \notin E(y)\end{cases}=\begin{cases} \varepsilon_{pu}&\text{ if }pu\in E(py)\\
0&\text{ if }pu\notin E(py)
\end{cases}
$$
when $p,u, y\in P$. Thus it suffices to prove \eqref{eq:Et-single} when $y=e$. Let
$y\in P$ such that $e\in E(y)$. Then $e\in E(y)=yE(e)\subseteq yP$. This forces $y\in P\cap P^{-1}$, so $y=e$. Hence $e\in E(e)$,
which also implies $p\in E(p)$ for all $p\in P$.
If $p\in E(e)$, then $p\in E(p)\cap E(e)$. This intersection
is non-empty precisely when $p=e$. In other words, we have established $E(e)=\{e\}$, from
which \eqref{eq:Et-single} and hence the theorem follow.
 \end{proof}

\begin{rmk}
It was asserted in \cite[\S 6.3, Remark 4]{N} that $\Kk(l^2(P))$ is an induced ideal
from $\Dd$ when $(G, P)$ has a \fp. However, no proof was given of this claim. Here we show that $\Kk(l^2(P))$ is
contained in the ideal of $\Tt(G, P)$ induced from $c_0(P)$. We conjecture that the two
are equal, but we have not been able to prove this.

To recall terminology, let $(G, P)$ be a quasi-lattice ordered group. Let $\Phi$ be the conditional expectation from $C^*(G, P)$
onto $B_P$ constructed in \cite{LacR1} and $\Phi^\normal$ the conditional expectation from
$\Tt(G, P)$ to $\Dd$ associated to the coaction
$\delta^\normal$ of Proposition~\ref{prop:normalising-delta}. The representation
$\lambda_T$ intertwines $\Phi$ and $\Phi^\normal$. Since $\delta^\normal$ is
normal, $\Phi^\normal$ is faithful on positive elements. In \cite[\S 6]{N} Nica associates to an invariant ideal $I$ in $\Dd$
the induced ideal  $\Ind{I}=\{X\in\Tt(G, P): \Phi^\normal(X^*X)\in \Dd\}$ in $\Tt(G, P)$.

Suppose $(G, P)$ has a \fp.  Lemma~\ref{lem:essential-ideal} says that
$\Ii:=c_0(\iota(P))$ is an essential ideal in
$\Dd$. Further, $\Ii$ is generated by the projections
$$
p_y=\prod_{a\in F} (T_yT_y^*-T_{ya}T_{ya}^*),
$$
for all $y\in P$. We claim that $\Ii$ is invariant in Nica's sense. To see this, let $x\in G$ and write it as $x=\sigma(x)\tau(x)^{-1}$
with $\sigma(x)\in P$ the least upper bound of $x$. For $y\in P$,  equation~\eqref{need-delta-inner} implies that
$$
T_{\sigma(x)}T_{\tau(x)}^*p_y (T_{\sigma(x)}T_{\tau(x)}^*)^*=p_{xy}T_{\sigma(x)}T_{\tau(x)}^*T_{\tau(x)}T_{\sigma(x)}^*,
$$
which is $p_{xy}T_{\sigma(x)}T_{\sigma(x)}^*$, and lies in $\Ii$ because $\Ii$ is an ideal in $\Dd$.
Since $p_y$ span $\Ii$, the ideal $\Ii$ is indeed invariant.

Now the rank-one operator on $l^2(P)$ taking $\varepsilon_y$ to $\varepsilon_x$ is $X=T_xp_eT_y^*$ and
$$
\Phi^\normal(X^*X)=T_yp_eT_y^*=p_y\in \Ii,
$$
so $\Ind \Ii$ contains all rank-one operators in $B(l^2(P))$. Hence $\Kk(l^2(P))\subset \Ind \Ii$.
\end{rmk}

\appendix
\section{Gauge-invariant uniqueness for Fell bundles}
 \label{appendix}

Here we present an abstract ``gauge-invariant uniqueness'' result for Fell bundles over discrete groups.
As applications we obtain gauge-invariant uniqueness results for maximal and for normal coactions.

\begin{prop}
\label{GIUTFB}
If $\pi:(A,\delta)\to (B,\varepsilon)$ is a surjective morphism of coactions such that
$\pi|_{A_e}$ is injective, then
\[
\pi\times G:A\times_\delta G\to B\times_\varepsilon G
\]
is an isomorphism.
Consequently:
\begin{enumerate}
\item\label{it:epsilon-max} if $\varepsilon$ is maximal, then $\delta$ is maximal and $\pi$ is an
isomorphism;

\item\label{it:delta-normal} if $\delta$ is normal, then $\varepsilon$ is normal and $\pi$ is an
isomorphism;

\item\label{it:delta-max} if $\delta$ is maximal, then $\pi$ is a maximalization of $(B,\varepsilon)$, and there is a unique morphism $\varphi:(B,\varepsilon)\to (A^n,\delta^n)$ such that
the diagram
\begin{equation}\label{A.1(3)}
\xymatrix{
(A,\delta) \ar[dr]^\pi \ar[dd]_{q_A^n}
\\
&(B,\varepsilon) \ar@{-->}[dl]^\varphi_{!}
\\
(A^n,\delta^n)
}
\end{equation}
commutes,
and moreover $\varphi$ is a normalization.

\item\label{it:epsilon-normal} if $\varepsilon$ is normal, then $\pi$ is a normalization of $(A,\delta)$, and there is a unique morphism $\varphi:(B^m,\varepsilon^m)\to (A,\delta)$ such that
the diagram
\begin{equation}\label{A.1(4)}
\xymatrix{
(B^m,\varepsilon^m) \ar@{-->}[dr]^\varphi_{!} \ar[dd]_{q_A^m}
\\
&(A,\delta) \ar[dl]^\pi
\\
(B,\varepsilon)
}
\end{equation}
commutes,
and moreover $\varphi$ is a maximalization.
\end{enumerate}
\end{prop}

\begin{proof}
We first show that
\[
\pi(A_s)=B_s\midtext{for all}s\in G.
\]
Indeed, it is easy to check on the generators that
\[
\varepsilon_s\circ\pi=\pi\circ\delta_s\midtext{for all}s\in G.
\]
Then we have
\begin{align*}
B_s
&=\varepsilon_s(B)
\\&=\varepsilon_s(\pi(A))
\\&=\pi(\delta_s(A))
\\&=\pi(A_s).
\end{align*}
Since $\pi|_{A_e}$ is injective, it follows that for each $s\in G$ the restriction $\pi|_{A_s}$ maps $A_s$ isometrically onto $B_s$,
and hence the associated Fell-bundle homomorphism $\wilde\pi:\Aa\to\Bb$ is an isomorphism.

The normalization
\[
\pi^n:(A^n,\delta^n)\to (B^n,\varepsilon^n)
\]
of $\pi$ is an isomorphism of coactions, because
$A^n\cong C^*_r(\Aa)$ and $B^n\cong C^*_r(\Bb)$.
Let $q^n_A:(A,\delta)\to (A^n,\delta^n)$ and $q^n_B:(B,\varepsilon)\to (B^n,\varepsilon^n)$ be the normalizing maps.

We have a commuting diagram
\[
\xymatrix{
(A,\delta) \ar[r]^-{q^n_A} \ar[d]_\pi
&(A^n,\delta^n) \ar[d]^{\pi^n}
\\
(B,\varepsilon) \ar[r]_-{q^n_B}
&(B^n,\varepsilon^n)
}
\]
of coaction morphisms, hence a commuting diagram
\[
\xymatrix{
A\times_\delta G \ar[r]^-{q^n_A\times G}_-\cong
\ar[d]_{\pi\times G}
&A^n\times_{\delta^n} G \ar[d]^{\pi^n\times G}_\cong
\\
B\times_\varepsilon G \ar[r]_-{q^n_B\times G}^\cong
&B^n\times_{\varepsilon^n} G
}
\]
of homomorphisms.
Thus $\pi\times G$ is an isomorphism.

Now (1)--(4) follow from the theory of maximalizations and normalizations:
First of all,
(1) and (2) follow immediately from \cite[Proposition~3.1]{clda}.

For (3), \cite[Proposition~6.1.11]{BKQ} shows that $\pi$ is a maximalization.
Let $q^n_B:(B,\varepsilon)\to (B^n,\varepsilon^n)$ be the normalization of $(B,\varepsilon)$.
Then $q^n_B\circ \pi:(A,\delta)\to (B^n,\varepsilon^n)$ also is a normalization, by \cite[Proposition~6.1.7]{BKQ}. Since all normalizations of $(A,\delta)$ are isomorphic, there is an isomorphism $\theta$ making the diagram
\[
\xymatrix@C+30pt{
(A,\delta) \ar[dr]^\pi \ar[dd]_{q_A^n}
\\
&(B,\varepsilon) \ar[d]^{q^n_B}
\\
(A^n,\delta^n)
&(B^n,\varepsilon^n) \ar@{-->}[l]^-\theta_-{\cong}
}
\]
commute.
Put $\varphi=\theta\circ q^n_B:(B,\varepsilon)\to (A^n,\delta^n)$.
Then $\varphi$ is a normalization since $q^n_B$ is and $\theta$ is an isomorphism,
and the diagram
\[
\xymatrix@C+30pt{
(A,\delta) \ar[dr]^\pi \ar[dd]_{q_A^n}
\\
&(B,\varepsilon) \ar[d]^{q^n_B} \ar[dl]_\varphi
\\
(A^n,\delta^n)
&(B^n,\varepsilon^n) \ar[l]^-\theta_-{\cong}
}
\]
commutes.

To see that $\varphi$ is the unique morphism making the diagram \eqref{A.1(3)} commute, suppose that $\varphi'$ is another. Since $q_A^n$ is also a maximalization (by \cite[Proposition~6.1.15]{BKQ}) it follows from the theory of maximalization that both $\varphi$ and $\varphi'$ have the same maximalization (namely $\id_A$), and hence are equal since the maximalization functor is faithful (by \cite[Corollary~6.1.19]{BKQ}).

(4) is proved similarly to (3): \cite[Proposition~6.1.7]{BKQ} shows that $\pi$ is a normalization, and if $q^m_A:(A^m,\delta^m)\to (A,\delta)$ is a maximalization then $\pi\circ q^m_A$ is also a maximalization, by \cite[Proposition~6.1.11]{BKQ}, so there is an isomorphism $\theta$ making the diagram
\[
\xymatrix@C+30pt{
(B^m,\varepsilon^m) \ar[dd]_{q_B^m} \ar@{-->}[r]^-\theta_-{\cong}
&(A^m,\delta^m) \ar[d]^{q^m_A}
\\
&(A,\delta) \ar[dl]^\pi
\\
(B,\varepsilon)
}
\]
commute.
Then $\varphi:=q^m_A\circ\theta$ is a maximalization of $(A,\delta)$ making the diagram
\[
\xymatrix@C+30pt{
(B^m,\varepsilon^m) \ar[dd]_{q_B^m} \ar[r]^-\theta_-{\cong} \ar[dr]_\varphi
&(A^m,\delta^m) \ar[d]^{q^m_A}
\\
&(A,\delta) \ar[dl]^\pi
\\
(B,\varepsilon)
}
\]
commute.

To prove that $\varphi$ is the unique morphism making the diagram \eqref{A.1(4)} commute, if $\varphi'$ is another then, since $q_B^m$ is also a normalization (by \cite[Proposition~6.1.14]{BKQ}) both $\varphi$ and $\varphi'$ have the same normalization (namely $\id_B$), and hence are equal since the normalization functor is faithful (by \cite[Corollary~6.1.19]{BKQ}).
\end{proof}

\begin{cor}[Abstract GIUT for maximal coactions]
\label{cor:giut-max-coact} Let $(A, \delta)$ be a maximal coaction and
$\pi:A\to B$  a surjective homomorphism. Then $\pi$ is
injective if and only if $\pi\vert_{A_e}$ is injective and there is a maximal
coaction $\varepsilon$ of $G$ on $B$ such that $\pi$ is $\delta-\varepsilon$ equivariant.
\end{cor}

\begin{proof} The forward direction is immediate. Assume now that $\pi\vert_{A_e}$ is injective
and there is a maximal coaction $\varepsilon$ of $G$ on $B$ such that $\pi$ is $\delta-\varepsilon$ equivariant.
Then $\pi: (A, \delta)\to (B, \varepsilon)$ is a surjective morphism of coactions. Hence $\pi$
is an isomorphism by Proposition~\ref{GIUTFB}, part~\eqref{it:epsilon-max}.
\end{proof}

The following is parallel to \corref{cor:giut-max-coact}:

\begin{cor}[Abstract GIUT for normal coactions]
\label{cor:giut-nor-coact} Let $(B, \varepsilon)$ be a normal coaction and
$\pi:A\to B$  a surjective homomorphism. Then $\pi$ is
injective if and only if there is a normal
coaction $\delta$ of $G$ on $A$ such that $\pi$ is $\delta-\varepsilon$ equivariant and
$\pi\vert_{A_e}$ is injective.
\end{cor}

\begin{proof} The forward direction is immediate. Assume now that there is a normal coaction
$\delta$ of $G$ on $A$ such that $\pi$ is $\delta-\varepsilon$ equivariant and
$\pi\vert_{A_e}$ is injective.
Then $\pi: (A, \delta)\to (B, \varepsilon)$ is a surjective morphism of coactions. Hence $\pi$
is an isomorphism by Proposition~\ref{GIUTFB}, part~\eqref{it:delta-normal}.
\end{proof}

\begin{cor}\label{cor:characterise-normal} Let $(A, \delta)$ be a coaction. The following are equivalent:
\begin{enumerate}
\item\label{it:delta-normal-cor} $\delta$ is normal;
\item\label{it:general-giut} A surjective homomorphism $\pi:A\to B$ is injective if and only if $\pi\vert_{A_e}$ is injective
and
there
is a coaction $\varepsilon$ on $B$ such that $\pi$ is $\delta-\varepsilon$ equivariant.
\end{enumerate}
\end{cor}

\begin{proof}
Assume \eqref{it:delta-normal-cor}. Let $\pi:A\to B$ be an isomorphism. Then trivially $\pi\vert_{A_e}$ is injective and $\pi$ carries
$\delta$ to a (normal) coaction on $B$. If on the other hand $\pi:A\to B$ is surjective, $\pi\vert_{A_e}$ is injective, and
$B$ carries a coaction $\varepsilon$ such that $\pi$ is $\delta-\varepsilon$ equivariant, then by Proposition~\ref{GIUTFB},
part~\eqref{it:delta-normal} $\pi$ is an isomorphism. This proves \eqref{it:delta-normal-cor}$\Rightarrow$\eqref{it:general-giut}.

Now assume \eqref{it:general-giut}. Since the normalization map $q_A^n:(A, \delta)\to (A^\normal, \delta^\normal)$
is equivariant and satisfies $q_A^n\vert_{A_e}$ is injective, by hypothesis $q_A^n$ is injective. Hence it is an isomorphism,
so $\delta$ is normal since $\delta^n$ is.
\end{proof}

The following is parallel to \corref{cor:characterise-normal}:

\begin{cor}\label{cor:characterise-max} Let $(B, \varepsilon)$ be a coaction. The following are equivalent:
\begin{enumerate}
\item\label{it:epsilon-max-cor} $\varepsilon$ is maximal;
\item\label{it:general-giut-2} A surjective homomorphism $\pi:A\to B$ is injective if and only if there
is a coaction $\delta$ on $A$ such that $\pi$ is $\delta-\varepsilon$ equivariant and
$\pi\vert_{A_e}$ is injective.
\end{enumerate}
\end{cor}

\begin{proof}
Assume \eqref{it:epsilon-max-cor}. Let $\pi:A\to B$ be an isomorphism. Then trivially
$\pi\inv$ carries $\varepsilon$ to a (maximal) coaction on $A$ and
$\pi\vert_{A_e}$ is injective. If on the other hand $\pi:A\to B$ is surjective,
$A$ carries a coaction $\delta$ such that $\pi$ is $\delta-\varepsilon$ equivariant and $\pi\vert_{A_e}$ is injective,
then by Proposition~\ref{GIUTFB},
part~\eqref{it:epsilon-max} $\pi$ is an isomorphism. This proves \eqref{it:epsilon-max-cor}$\Rightarrow$\eqref{it:general-giut-2}.

Now assume \eqref{it:general-giut-2}. Since the maximalization map $q_B^m:(B^m, \varepsilon^m)\to (B, \varepsilon)$
is equivariant and satisfies $q_B^m\vert_{B^m_e}$ is injective, by hypothesis $q_B^m$ is injective. Hence it is an isomorphism,
so $\varepsilon$ is maximal since $\varepsilon^m$ is.
\end{proof}


\begin{thebibliography}{EKQR06}

\bibitem[BKQ11]{BKQ}
E.~B{\'e}dos, S.~Kaliszewski, and J.~Quigg, \emph{Reflective-coreflective
  equivalence}, Theory Appl. Categ. \textbf{25} (2011), 142--179.

\bibitem[BaHLR]{BanHLR}
N.~Brownlowe, A.~an~Huef, M.~Laca, and I.~Raeburn, \emph{Boundary quotients of
  the {T}oeplitz algebra of the affine semigroup over the natural numbers},
  Ergodic Theory \& Dynam. Systems, (1) \textbf{32} (2012), 35--62.

\bibitem[CLSV]{CLSV}
T.M. Carlsen, N.S. Larsen, A.~Sims, and S.T. Vittadello, \emph{Co-universal
  algebras associated to product systems, and gauge-invariant uniqueness
  theorems}, Proc. London Math. Soc., (4) \textbf{103} (2011), 563--600.

\bibitem[Cob67]{Co}
L.A. Coburn, \emph{The ${C}^*$-algebra generated by an isometry {I}}, Bull.
  Amer. Math. Soc. \textbf{73} (1967), 722--726.

\bibitem[CL02]{CLac1}
J.~Crisp and M.~Laca, \emph{On the {T}oeplitz algebras of right-angled and
  finite-type {A}rtin groups}, J. Austral. Math. Soc. \textbf{72} (2002),
  223--245.

\bibitem[Cun77]{Cu}
J.~Cuntz, \emph{Simple ${C}^*$-algebras generated by isometries}, Comm. Math.
  Phys. \textbf{57} (1977), 173--185.

\bibitem[EKQ04]{EKQ}
S.~Echterhoff, S.~Kaliszewski, and J.~Quigg, \emph{Maximal coactions},
  Internat. J. Math. \textbf{15} (2004), 47--61.

\bibitem[EKQR06]{BE}
S.~Echterhoff, S.~Kaliszewski, J.~Quigg, and I.~Raeburn, \emph{{A Categorical
  Approach to Imprimitivity Theorems for C*-Dynamical Systems}}, vol. 180, Mem.
  Amer. Math. Soc., no. 850, American Mathematical Society, Providence, RI,
  2006.

\bibitem[EQ99]{EQ}
S.~Echterhoff and J.~Quigg, \emph{Induced coactions of discrete groups on
  {$C^*$}-algebras}, Canad. J. Math. \textbf{51} (1999), 745--770.

\bibitem[Exe97]{E}
R.~Exel, \emph{Amenability for {F}ell bundles}, J. Reine Angew. Math.
  \textbf{492} (1997), 41--73.

\bibitem[Fow02]{F}
N.~J. Fowler, \emph{Discrete product systems of {H}ilbert bimodules}, Pacific
  J. Math. \textbf{204} (2002), 335--375.

\bibitem[aHR97]{anHR}
A.~an~Huef and I.~Raeburn, \emph{The ideal structure of {C}untz-{K}rieger
  algebras}, Ergodic Theory \& Dynam. Systems \textbf{17} (1997), 611--624.

\bibitem[Kat04]{Ka}
T.~Katsura, \emph{On {$C^*$}-algebras associated with {$C^*$}-correspondences},
  J. Funct. Anal. \textbf{217} (2004), 366--401.

\bibitem[KQ09]{clda}
S.~Kaliszewski and J.~Quigg, \emph{Categorical {L}andstad duality for actions},
  Indiana Univ. Math. J. \textbf{58} (2009), no.~1, 415--441.

\bibitem[KQ10]{nordfjordeid}
\bysame, \emph{Categorical perspectives on noncommutative duality}, Summer
  school on C*-algebras and their interplay with dynamical systems, June 2010,
  http://www.ntnu.no/imf/english/research/fa/oa/summerschool2010.

\bibitem[Lac99]{Lac1}
M.~Laca, \emph{Purely infinite simple {T}oeplitz algebras}, J. Operator Theory
  \textbf{41} (1999), 421--435.

\bibitem[LR96]{LacR1}
M.~Laca and I.~Raeburn, \emph{{Semigroup crossed products and {T}oeplitz
  algebras of nonabelian groups}}, J. Funct. Anal. \textbf{139} (1996),
  415--440.

\bibitem[Nic92]{N}
A.~Nica, \emph{{{$C^*$}-algebras generated by isometries and {W}iener-{H}opf
  operators}}, J. Operator Theory \textbf{27} (1992), 17--52.

\bibitem[Pim97]{P}
M.V. Pimsner, \emph{A class of {$C^*$}-algebras generalizing both
  {C}untz-{K}rieger algebras and crossed products by {${\bf Z}$}}, Free
  probability theory(Waterloo, ON, 1995), Amer. Math. Soc., Providence, RI,
  1997, pp.~189--212.

\bibitem[Qui94]{QuiggFull}
J.~C. Quigg, \emph{{Full and reduced {$C^*$}-coactions}}, Math. Proc. Camb.
  Phil. Soc. \textbf{116} (1994), 435--450.

\bibitem[Qui96]{QuiggDiscrete}
\bysame, \emph{{Discrete {$C^*$}-coactions and {$C^*$}-algebraic bundles}}, J.
  Austral. Math. Soc. Ser. A \textbf{60} (1996), 204--221.

\bibitem[QR97]{QuiggRa}
J.~Quigg and I.~Raeburn, \emph{{Characterizations of crossed products by
  partial actions}}, J. Operator Theory \textbf{37} (1997), 311--340.

\bibitem[SY10]{SY}
A.~Sims and T.~Yeend, \emph{{$C^*$}-algebras associated to product systems of
  {H}ilbert bimodules}, J. Operator Theory \textbf{64} (2010), 349--376.

\end{thebibliography}
\providecommand{\bysame}{\leavevmode\hbox to3em{\hrulefill}\thinspace}
\providecommand{\MR}{\relax\ifhmode\unskip\space\fi MR }
\providecommand{\MRhref}[2]{%
  \href{http://www.ams.org/mathscinet-getitem?mr=#1}{#2}
}
\providecommand{\href}[2]{#2}

\end{document}